\documentclass{siamltex}
\usepackage[numbers]{natbib}
\usepackage{amsmath}
\usepackage{amscd}
\usepackage{amssymb}
\usepackage{graphicx}
\usepackage{subcaption}
\usepackage{url}
\usepackage{color}
\usepackage{listings}
\usepackage{booktabs}
\usepackage{caption}
\usepackage{textcomp}
\usepackage{mathtools}
\usepackage{mcode}
\usepackage{hyperref}
\newcommand{\mpd}[1]{\textrm{mpd}(#1)}

\renewcommand{\rank}[1]{\mathrm{rank}(#1)}
\newcommand{\dk}[1]{\textrm{dim ker}(#1)}

\newtheorem{remark}{Remark}

\definecolor{DarkBlue}{rgb}{0.00,0.00,0.55}
\definecolor{Black}{rgb}{0.00,0.00,0.00}
\hypersetup{
    linkcolor = DarkBlue,
    anchorcolor = DarkBlue,
    citecolor = DarkBlue,
    filecolor = DarkBlue,
    pagecolor = DarkBlue,
    urlcolor = DarkBlue,
    colorlinks  = true,
}
\title{The number of distinct eigenvalues of a matrix after perturbation}

\author{
  P.~E.~Farrell\thanks{Mathematical Institute, University of Oxford, Oxford, UK.
    Center for Biomedical Computing, Simula Research Laboratory, Oslo, Norway
    (\texttt{patrick.farrell@maths.ox.ac.uk}).
    This research is funded by EPSRC grants EP/K030930/1, EP/M019721/1, 
    and a Center of Excellence grant from the Research
    Council of Norway to the Center for Biomedical Computing at Simula
    Research Laboratory. The author would like to thank
    A.~J.~Wathen, M.~A.~Saunders,
    C.~Beentjes, and N.~I.~M.~Gould for useful discussions.
    }
  }

\begin{document}
\maketitle

\begin{abstract}
We prove a new theorem relating the number of distinct eigenvalues of a matrix
after perturbation to the prior number of distinct eigenvalues, the rank of the
update, and the degree of nondiagonalizability of the matrix. In particular, a
rank one update applied to a diagonalizable matrix can at most double the number
of distinct eigenvalues. The theorem applies to both symmetric and nonsymmetric
matrices and perturbations, of arbitrary magnitudes. An an application, we prove that in
exact arithmetic the number of
Krylov iterations required to exactly solve a linear system involving a diagonalizable
matrix can at most double after a rank one update.
\end{abstract}

\begin{keywords}
distinct eigenvalues, perturbation, Krylov methods, deflation.
\end{keywords}

\begin{AMS}
15A18, 65F15, 65F10
\end{AMS}

\section{Distinct eigenvalues after perturbation}
The spectrum of a matrix after perturbation is of interest in a wide variety
of applications and has been studied extensively in various particular cases,
with most work focussing on the case of symmetric rank one perturbations
\cite{wilkinson1965,golub1973,bunch1978,ipsen2009}. More general results
concern the Jordan form of the matrix after ``generic'' rank one
perturbations, i.e.~the set of rank one perturbations for which the analysis
does not hold has Lebesgue measure zero
\cite{hormander1994,savchenko2003,moro2003,mehr2013}. In this work we prove a
new theorem regarding the number of distinct eigenvalues of arbitrary matrices
perturbed by updates of arbitrary rank.

Let $\Lambda(M)$ be the set of distinct eigenvalues of a matrix $M$. Let $m_a(M,
\lambda)$ and $m_g(M, \lambda)$ be the algebraic and geometric multiplicity of $\lambda$ as an eigenvalue of $M$,
respectively.
\begin{definition}[Defectivity of an eigenvalue] \label{def:defectivity}
The defectivity of an eigenvalue $d(M, \lambda) \ge 0$ is the difference between
its algebraic and geometric multiplicities,
\begin{equation}
d(M, \lambda) \equiv m_a(M, \lambda) - m_g(M, \lambda).
\end{equation}
\end{definition}
\begin{definition}[Defectivity of a matrix]
The defectivity of a matrix $d(M)$ is the sum of the defectivities of its
eigenvalues:
\begin{equation}
d(M) \equiv \sum_{\lambda \in \Lambda(M)} \big( m_a(M, \lambda) - m_g(M, \lambda) \big).
\end{equation}
\end{definition}
Recall that $m_a(M, \lambda) \ge m_g(M, \lambda)$ for all $M$ and $\lambda$. Thus,
$d(M, \lambda) \ge 0$, and $d(M) \ge 0$.
Defectivity is a quantitative measure of nondiagonalizability: a matrix
is diagonalizable if and only if it has defectivity zero.
\begin{remark}
The defectivity of a matrix is clear from its Jordan form: it is the number of
off-diagonal ones.
\end{remark}

We now give the central theorem of this paper.
\begin{theorem} \label{thm:fullthm}
Let $A, B \in \mathbb{C}^{n \times n}$. If $C = A + B$, then $\left|\Lambda(C)\right| \le (\rank{B}+1) \left|\Lambda(A)\right| + d(A)$.
\end{theorem}
\begin{proof}
Clearly $\left|\Lambda(C)\right| = \left|\Lambda(C) \cap \Lambda(A)\right| + \left|\Lambda(C)
\setminus \Lambda(A)\right|$, and the first term is bounded by $\left|\Lambda(A)\right|$. We seek an
upper bound for the quantity
\begin{equation}
\sum_{\begin{subarray}

\lambda \in \Lambda(C) \\
\lambda \notin \Lambda(A) \\
\end{subarray}} m_a(C, \lambda)
\end{equation}
as this bounds the number of new eigenvalues that the perturbation can introduce. (Every
eigenvalue $\lambda$ of $C$ must have $m_a(C, \lambda) \ge 1$.) Since for $M \in \mathbb{R}^{n \times n}$,
\begin{equation}
\sum_{\lambda \in \Lambda(M)} m_a(M, \lambda) = n,
\end{equation}
it follows that
\begin{equation}
\sum_{\begin{subarray}

\lambda \in \Lambda(C) \\
\lambda \notin \Lambda(A) \\
\end{subarray}} m_a(C, \lambda) + \sum_{\lambda \in \Lambda(A)} m_a(C, \lambda) = n,
\end{equation}
with the convention that $m_a(C, \lambda) = 0 \iff \lambda \notin \Lambda(C)$. Thus, the upper bound on the number of new eigenvalues
introduced is maximized when
\begin{equation}
\sum_{\lambda \in \Lambda(A)} m_a(C, \lambda)
\end{equation}
is minimized.

Let $\lambda \in \Lambda(A)$. We first
investigate $m_g(C, \lambda)$, the geometric multiplicity of $\lambda$ as an
eigenvalue of the perturbed matrix $C$. Using the fact that $\rank{X + Y} \le
\rank{X} + \rank{Y}$, we derive a lower bound for $m_g(C, \lambda)$:
\begin{subequations}
\begin{alignat}{2}
&\quad &\rank{A + B - \lambda I} &\le \rank{A - \lambda I} + \rank{B} \\
&\implies &n - \dk{A + B - \lambda I} & \le n - \dk{A - \lambda I} + \rank{B} \label{eqn:rankB} \\
&\implies &m_g(C, \lambda) &\ge m_g(A, \lambda) - \rank{B}. \label{eqn:rank_lowerbound}
\end{alignat}
\end{subequations}
Hence, the geometric multiplicity of an eigenvalue can at most decrease
by $r$ on perturbation by a rank-$r$ operator.

It therefore follows that
\begin{subequations} \label{eqn:allinequality}
\begin{align}
\sum_{\lambda \in \Lambda(A)} m_a(C, \lambda) &\ge \sum_{\lambda \in \Lambda(A)} m_g(C, \lambda) \\
                                              &\ge \sum_{\lambda \in \Lambda(A)} \big(m_g(A, \lambda) - \rank{B}\big) \qquad \qquad \qquad \ \text{ (by \eqref{eqn:rank_lowerbound})} \label{eqn:beforediagonalisability} \\
                                              &=   \sum_{\lambda \in \Lambda(A)} \big(m_a(A, \lambda) - \rank{B} - d(A, \lambda)\big) \qquad \text{ (by Definition \ref{def:defectivity})} \label{eqn:diagonalisability} \\
                                              &=   n - \rank{B}\left|\Lambda(A)\right| - d(A). \label{eqn:lowerbound}
\end{align}
\end{subequations}

The maximal number of new eigenvalues is achieved when
\eqref{eqn:allinequality} is an equality,
and
\begin{equation}
\sum_{\begin{subarray}

\lambda \in \Lambda(C) \\
\lambda \notin \Lambda(A) \\
\end{subarray}} m_a(C, \lambda) = n - \left(n - \rank{B}\left|\Lambda(A)\right| - d(A)\right) = \rank{B}\left|\Lambda(A)\right| + d(A).
\end{equation}
Hence
\begin{align}
\left|\Lambda(C)\right| &= \left|\Lambda(C) \cap \Lambda(A)\right| + \left|\Lambda(C) \setminus \Lambda(A)\right| \nonumber\\
                                  &\le \left|\Lambda(A)\right| + \rank{B}\left|\Lambda(A)\right| + d(A) = (\rank{B}+1)\left|\Lambda(A)\right| + d(A).
\end{align}
\end{proof}

\begin{corollary}
Let $A$ be diagonalizable $($i.e., $d(A) = 0)$ and let $B$ have rank one. If $C = A + B$, then $\left|\Lambda(C)\right| \le 2 \left|\Lambda(A)\right|$.
\end{corollary}

\section{Krylov iterations after a rank one update}
Consider the linear systems $Ax = b$ and $Cy = d$.
If $A$ is diagonalizable, then its minimal polynomial degree $\mpd{A} =
\left|\Lambda(A)\right|$, and an optimal Krylov method (GMRES \cite{saad1986},
MINRES \cite{paige1975}, or CG \cite{hestenes1952}, if applicable) will
compute $x$ exactly in the same number of iterations. (Here, and henceforth, exact
arithmetic is assumed.)

\begin{theorem} \label{thm:krylov}
Consider the linear systems $Ax = b$ and $Cy = d$.
Let $A$ be diagonalizable, and let $B$ have rank one. If $C = A + B$, then $y$
can be computed exactly with an optimal Krylov method in at most double the
number of iterations required for $x$.
\end{theorem}
\begin{proof}
If $C$ is diagonalizable, then $\mpd{C} = \left|\Lambda(C)\right| \le
2\left|\Lambda(A)\right| = 2\ \mpd{A}$, i.e.~the number of distinct eigenvalues
bounds the number of Krylov iterations required to solve the perturbed matrix.

If $C$ is not diagonalizable, we know from \eqref{eqn:rank_lowerbound} that the
number of Jordan blocks associated with an eigenvalue $\lambda \in \Lambda(A)$
can decrease by at most $1 = \rank{B}$ in $C$. Since by diagonalizability of $A$
all its Jordan blocks are of size $1 \times 1$, the largest Jordan block of
$C$ can be at most of size $\left|\Lambda(A)\right| \times
\left|\Lambda(A)\right|$, which can occur when all eigenvalues of $A$ lose exactly
one Jordan block. It is straightforward to calculate that with any
arrangement of new Jordan blocks of $C$ with sizes adding to $\left|\Lambda(A)\right|$,
the number of Krylov iterations required to compute $y$ is bounded by twice that
of $x$.
\end{proof}

\section{Application: Schur complement preconditioners and deflation}
Theorem \ref{thm:fullthm} is mainly of interest in situations where
$|\Lambda(A)|$ is expected to be small. Such a situation arises in the
application of preconditioners based on Schur complements.

Let $F: \mathbb{R}^n \to \mathbb{R}^n$ be the (discretized) residual of a nonlinear problem
\begin{equation} \label{eqn:problem}
F(u) = 0
\end{equation}
with block-structured Jacobian
\begin{equation}
J =
\begin{bmatrix}
X & Y \\
Z & 0
\end{bmatrix},
\end{equation}
with $X$ invertible. This structure arises in many problems, including the
Stokes and Navier--Stokes equations, and in equality-constrained optimization
\cite{benzi2005}. Linear systems involving $J$ are typically solved with Schur
complement preconditioners. Define
\begin{equation}
P = 
\begin{bmatrix}
X & 0 \\
0 & -S
\end{bmatrix},
\end{equation}
where the Schur complement $S = -ZX^{-1}Y$. If $P$ is used as a preconditioner,
then the preconditioned operator $P^{-1}J$ is diagonalizable and has exactly
three distinct eigenvalues (with exact inner solves for the application of
$P^{-1}$) \cite{murphy2000}. Similar results hold for more general
block-structured Jacobians and preconditioners based on the Schur
complement: the preconditioned operator has a small number of distinct
eigenvalues \cite{ipsen2001}.

Suppose \eqref{eqn:problem} supports multiple solutions. One approach to compute them is to
initialize Newton's method from many different initial guesses, hoping to start
in different basins of convergence. A highly effective alternative is to
\emph{deflate} known solutions \cite{farrell2014}. Suppose one solution $u_1^*$
of \eqref{eqn:problem} has been computed from an initial guess $u_0$ and
additional solutions are sought. We construct a modified residual
\begin{equation} \label{eqn:deflatedproblem}
G(u) = M(u; u_1^*) F(u),
\end{equation}
via the application of a \emph{deflation operator} $M: \mathbb{R}^n \times \mathbb{R}^n \to \mathbb{R}$ to the residual
$F$. This deflation operator guarantees two properties. The first is the
preservation of solutions of $F$, i.e.  for $u \neq u_1^*$, $G(u) = 0 \iff F(u)
= 0$. The second is that Newton's method (or other rootfinding algorithms)
applied to $G$ will not discover $u_1^*$ again, as
\begin{equation}
\liminf_{u \rightarrow u_1^*} \|G(u)\| > 0,
\end{equation}
i.e.~along any sequence converging to the known root, its existence is masked by
the nonconvergence of the deflated residual to zero. ($M$ achieves this by introducing
a pole of the appropriate strength at the known solution.) Thus, if Newton's method
applied to $G$ converges from $u_0$, it will converge to another solution $u_2^*
\neq u_1^*$. A typical deflation operator is
\begin{equation}
M(u; u_1^*) = \frac{1}{\|u - u_1^*\|^p} + 1,
\end{equation}
where $p$ controls the strength of the pole introduced.

The process can then be repeated until no more solutions are found
from $u_0$ in a fixed number of Newton iterations. Several solutions can
be deflated with an operator $M: \mathbb{R}^n \times \mathbb{R}^n \times \cdots \mathbb{R}^n \to \mathbb{R}$ via
\begin{equation}
M(u; u_1^*, \dots, u_k^*) = \frac{1}{\|u - u_1^*\|^p} \cdots \frac{1}{\|u - u_k^*\|^p} + 1.
\end{equation}
For full details, see Brown
and Gearhart \cite{brown1971} and Farrell
et al.~\cite{farrell2014}.

The Jacobian $\tilde{J}$ of the deflated problem \eqref{eqn:deflatedproblem} is a rank one
update of a scaling of the Jacobian of the original problem \eqref{eqn:problem}, regardless of
the number of solutions deflated:
\begin{equation}
\tilde{J} = M J + F E^T,
\end{equation}
where $E = M' \in \mathbb{R}^n$.
Hence, the preconditioned deflated Jacobian is also a rank one update of the
preconditioned original Jacobian,
\begin{equation}
C = P^{-1} \tilde{J} = M P^{-1} J + (P^{-1} F) E^T = A + B,
\end{equation}
and Theorem \ref{thm:krylov} guarantees that the solutions of linear systems involving
the deflated Jacobian can be computed exactly in no more than twice the number
of Krylov iterations required for the undeflated Jacobian.

\bibliographystyle{siam}
\bibliography{literature}

\begin{thebibliography}{10}

\bibitem{benzi2005}
{\sc M.~Benzi, G.~H. Golub, and J.~Liesen}, {\em Numerical solution of saddle
  point problems}, Acta Numerica, 14 (2005), pp.~1--137.

\bibitem{brown1971}
{\sc K.~M. Brown and W.~B. Gearhart}, {\em Deflation techniques for the
  calculation of further solutions of a nonlinear system}, Numerische
  Mathematik, 16 (1971), pp.~334--342.

\bibitem{bunch1978}
{\sc J.~R. Bunch, C.~P. Nielsen, and D.~C. Sorensen}, {\em Rank-one
  modification of the symmetric eigenproblem}, Numerische Mathematik, 31
  (1978), pp.~31--48.

\bibitem{farrell2014}
{\sc P.~E. Farrell, \'A. Birkisson, and S.~W. Funke}, {\em Deflation techniques
  for finding distinct solutions of nonlinear partial differential equations},
  SIAM Journal on Scientific Computing, 37 (2015), pp.~A2026--A2045.

\bibitem{golub1973}
{\sc G.~H. Golub}, {\em Some modified matrix eigenvalue problems}, SIAM Review,
  15 (1973), pp.~318--334.

\bibitem{hestenes1952}
{\sc M.~R. Hestenes and E.~Stiefel}, {\em Methods of {C}onjugate {G}radients
  for solving linear systems}, Journal of Research of the National Bureau of
  Standards, 49 (1952), pp.~409--436.

\bibitem{hormander1994}
{\sc L.~H\"ormander and A.~Melin}, {\em A remark on perturbations of compact
  operators}, Mathematica Scandinavica, 75 (1994), pp.~255--262.

\bibitem{ipsen2001}
{\sc I.~C.~F. Ipsen}, {\em A note on preconditioning nonsymmetric matrices},
  SIAM Journal on Scientific Computing, 23 (2001), pp.~1050--1051.

\bibitem{ipsen2009}
{\sc I.~C.~F. Ipsen and B.~Nadler}, {\em Refined perturbation bounds for
  eigenvalues of {H}ermitian and non-{H}ermitian matrices}, SIAM Journal on
  Matrix Analysis and Applications, 31 (2009), pp.~40--53.

\bibitem{mehr2013}
{\sc C.~Mehl, V.~Mehrmann, A.~C.~M. Ran, and L.~Rodman}, {\em Jordan forms of
  real and complex matrices under rank one perturbations}, Operators and
  Matrices, 7 (2013), pp.~381--398.

\bibitem{moro2003}
{\sc J.~Moro and F.~M. Dopico}, {\em Low rank perturbation of {J}ordan
  structure}, SIAM Journal on Matrix Analysis and Applications, 25 (2003),
  pp.~495--506.

\bibitem{murphy2000}
{\sc M.~F. Murphy, G.~H. Golub, and A.~J. Wathen}, {\em A note on
  preconditioning for indefinite linear systems}, SIAM Journal on Scientific
  Computing, 21 (2000), pp.~1969--1972.

\bibitem{paige1975}
{\sc C.~Paige and M.~Saunders}, {\em Solution of sparse indefinite systems of
  linear equations}, SIAM Journal on Numerical Analysis, 12 (1975),
  pp.~617--629.

\bibitem{saad1986}
{\sc Y.~Saad and M.~Schultz}, {\em {GMRES}: a generalized minimal residual
  algorithm for solving nonsymmetric linear systems}, SIAM Journal on
  Scientific and Statistical Computing, 7 (1986), pp.~856--869.

\bibitem{savchenko2003}
{\sc S.~V. Savchenko}, {\em Typical changes in spectral properties under
  perturbations by a rank-one operator}, Mathematical Notes, 74 (2003),
  pp.~557--0568.

\bibitem{wilkinson1965}
{\sc J.~H. Wilkinson}, {\em The Algebraic Eigenvalue Problem}, vol.~87 of
  Monographs on Numerical Analysis, Oxford University Press, 1965.

\end{thebibliography}

\end{document}